\documentclass[letterpaper, 10 pt, conference]{ieeeconf}  

\IEEEoverridecommandlockouts                              
\usepackage{graphics} 
\usepackage{epsfig} 
\usepackage{amsmath} 
\usepackage{amssymb}  
\usepackage{listings}
\usepackage{epstopdf}
\usepackage{algorithmic,algorithm}
\usepackage{color}

\newcommand{\mR}{{\mathbb R}}

\newcommand{\Hb}{\mathbf{H}}

\newcommand{\rank}{{\rm rank}}

\title{\LARGE \bf
Spectral Rank, Feedback, Causality and the\\ Indirect Method for CARMA Identification
}

\author{Wenqi Cao$^{1}$, Anders Lindquist$^{2}$ and Giorgio Picci$^{3}$
\thanks{$^{1}$Department of Automation, Shanghai Jiao Tong University, Shanghai, China. {\tt\small wenqicao@sjtu.edu.cn}}%
\thanks{$^{2}$Department of Automation and School of Mathematical Sciences, Shanghai
Jiao Tong University, Shanghai, China. {\tt\small alq@math.kth.se}}%
\thanks{$^{3}$Department of Information Engineering, University of Padova, Italy. {\tt\small picci@dei.unipd.it}%
}
}

\begin{document}

\maketitle
\thispagestyle{empty}
\pagestyle{empty}

\newtheorem{definition}{\bf Definition}
\newtheorem{lemma}{\bf Lemma}
\newtheorem{theorem}{\bf Theorem}
\newtheorem{example}{\bf Example}
\newtheorem{corollary}{\bf Corollary}
\newtheorem{remark}{\bf Remark}
\newtheorem{proposition}{\bf Proposition}

\begin{abstract}
Building on a recent paper by Georgiou and Lindquist [1] on the problem of rank deficiency of spectral densities and hidden dynamical relations after sampling of continuous-time stochastic processes, this paper is devoted to understanding related questions of feedback and Granger causality that affect stability properties. This then naturally  connects to CARMA identification, where we remark on certain oversights in the literature.
\end{abstract}


\section{Introduction}\label{sec:intro}

Consider a continuous-time stationary $(p+q)$-dimensional vector process $\{\zeta(t); t\in(-\infty,\infty)\}$ with a strictly proper rational spectral density, where each component $\zeta_k$, $k=1,2,\dots,p+q$, corresponds to a node in a dynamical network.  Georgiou and Lindquist considered in \cite{GLsampling} the problem of finding deterministic dynamical relations between subvectors of components of $\zeta$ from sampled data. More precisely, possibly after rearranging the order of the components of $\zeta$, let
\begin{equation}\label{eq:z}
\zeta(t):= \begin{bmatrix}y(t)\\u(t)\end{bmatrix},
\end{equation}
where $y=(\zeta_1,\dots,\zeta_p)'$ and $u:=(\zeta_{p+1},\dots,\zeta_{p+q})'$ are stochastic vector processes of dimensions $p$ and $q$, respectively, and prime denotes transposition.  The question addressed in \cite{GLsampling} is then whether there exists a proper rational transfer function $F(s)$ as in Figure~\ref{system} mapping the signal $u$ to $y$ and whether this  transfer function can be recovered from sampled data.
\begin{figure}[h]
      \centering
      \includegraphics[scale=0.45]{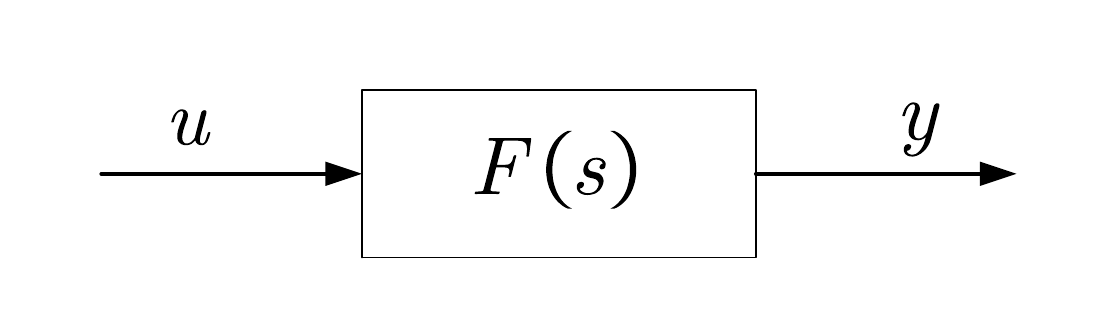}
      \caption{Dynamical relation from $u$ to $y$}
      \label{system}
\end{figure}
In general, such an $F(s)$ would not be stable since the complete system may be stabilized via feedback in the network, but conditions for stability was left as an open question in \cite{GLsampling}.

As in \cite{GLsampling}, we assume that the stationary process $\zeta$ has a minimal Markovian representation  \cite{LPbook}
\begin{subequations}\label{eq:continuous}
\begin{align}
\label{eq:continuous_model}
dx(t)&=Ax(t)dt + Bdw(t)\\
\label{eq:dx}
\zeta(t)&=Cx(t)
\end{align}
\end{subequations}
with $C\in \mR^{(p+q)\times n}$, $A\in \mR^{n\times n}$, $B\in \mR^{n\times m}$ and $w(t)$ a standard $m$-dimensional vector Wiener process. Moreover,  $A$ is a stability matrix, i.e., having all eigenvalues in the open left half plane, and $(C,A)$ and $(A,B)$ are  observable and reachable pairs, respectively. We also assume that $B$ has full column rank and that $\rank(CB)=m$. Then $\zeta$ has  a $(p+q)\times (p+q)$ rational spectral density
\begin{subequations}\label{Phizeta}
\begin{equation}\label{eq:factorization}
\Phi_\zeta(i\omega)=W(i\omega)W(-i\omega)^\prime ,
\end{equation}
where we take $m$ to be the rank of $\Phi_\zeta$ and
\begin{equation}
\label{V}
W(s)=C(sI-A)^{-1}B ,
\end{equation}
\end{subequations}
i.e., $\zeta$ has the spectral representation
\begin{subequations}\label{sprepr}
\begin{equation}
\label{zetaspecrepr}
\zeta(t)=\int_{-\infty}^\infty e^{i\omega t}d\hat{\zeta}=\int_{-\infty}^\infty e^{i\omega t}W(i\omega)d\hat{w},
\end{equation}
where $d\hat{\zeta}$ is an orthogonal stochastic measure such that $\mathbb{E}\{d\hat{\zeta}d\hat{\zeta}^*\}=\Phi_{\zeta}d\omega$ (e.g., \cite[p. 77]{LPbook}) and thus
\begin{equation}
\label{dwhat}
\mathbb{E}\{d\hat{w}d\hat{w}^*\}=\frac{d\omega}{2\pi} I.
\end{equation}
\end{subequations}

To better understand the relation between the processes $u$ and $y$ and the transfer function $F(s)$, in this paper we study the feedback structure provided by the overall network and its relation to spectral rank in a wider context of feedback models.  For example, when is the transfer function stable and when does the network provide feedback? This leads to the study of Granger causality \cite{Granger63,Granger69}, which is increasingly applied in system science and machine learning \cite{Gong15,JPC19,BarnettSeth15,BS17}. A stronger type of causality is when $F(s)$ is actually stable and no feedback is needed. In \cite{Caines} and \cite{LPbook} causality is discussed in the context of linear stochastic system.

This paper is devoted to continuous-time processes, but similar result can be derived for discrete-time models \cite{WenqiThesis}; also see \cite{GLsampling}. Processes with singular spectral densities play an important role in dynamic factor models where there are a latent processes with a singular spectral densities.  In econometrics singular AR and ARMA systems are important in the context of both dynamic factor models and DSGE  (dynamic stochastic general equilibrium) models \cite{EJC,Deistler19}.

One main issue in \cite{GLsampling} was how to restore the continuous-time model \eqref{eq:continuous} and its possible rank deficiency from sampled data. In hindsight it is realized that the identification part of this question had been studied before in the context of the so called {\em Indirect method of CARMA identification\/} \cite{Soderstrom91,MahataFu07,BPS13}, where CARMA stands for Continuous-time Autoregressive  Moving Average. However, starting from a completely new perspective, \cite{GLsampling} provides insights that were overlooked in the literature on CARMA identification. It turns out that rank deficiencies that are hidden in the sampled version of the system are crucial in formulating the identification problem correctly.

The structure of the paper is as follows. In Section~\ref{sec:feedback} we formulate a general feedback model for a pair of jointly correlated stationary processes and investigate related questions of spectral rank. This is then specialized in Section~\ref{sec:determ} to the situation illustrated by Figure~\ref{system}. Granger causality and its connection to feedback is introduced in Section~\ref{sec:Granger}. Section~\ref{sec:examples} provides some illustrative examples. In Section~\ref{sec:CARMA} we remark on the question of rank deficiencies that are hidden in the sampled system. Finally, in Section~\ref{sec:conclusions} we provide some conclusions.

\section{Feedback models and spectral rank}\label{sec:feedback}

The stochastic processes $y$ and $u$ introduced in Section~\ref{sec:intro}  are jointly stationary, and the rational spectral density \eqref{eq:factorization} has the natural partitioning
\begin{equation}
\label{Phidecomp}
\Phi_\zeta(s)=\begin{bmatrix}
\Phi_y(s)&\Phi_{yu}(s)\\ \Phi_{uy}(s)&\Phi_u(s)
\end{bmatrix}.
\end{equation}
Let $\Hb_t^-(u)$ be the closed span of the past components $\{ u_1(\tau), u_2(\tau), \dots, u_q(\tau)\}\mid \tau \leq t\}$ of the vector process $u$ in the Hilbert space of random variables, and let $\Hb_t^-(y)$ be defined likewise in terms of $\{ y_1(\tau), y_2(\tau), \dots, y_p(\tau)\mid \tau \leq t\}$. For future use, we shall also need the closed span $\Hb_t^+(u)$ of the future components $\{ u_1(\tau), u_2(\tau), \dots, u_q(\tau)\}\mid \tau \geq t\}$  and the closed span $\Hb(u)$ of the complete (past and future) history of $u$, and similarly for $y$.

We can now express both $y(t)$ and $u(t)$ as a sum of the best linear estimate based on the past of the other process plus an error term, i.e.,
 \begin{subequations}\label{yu}
           \begin{align}
    y(t) &= \mathbb{E}\{y(t)\mid \Hb_t^-(u)\} + v(t),\label{u2y}\\
    u(t) &= \mathbb{E}\{u(t)\mid \Hb^-_{t}(y)\} + r(t).\label{y2u}
\end{align}
\end{subequations}
The two error processes  $v(t)$ and $r(t)$ are jointly stationary with spectral representations
\begin{equation}
\label{vr}
v(t)=\int_{-\infty}^\infty e^{i\omega t}d\hat{v}, \quad r(t)=\int_{-\infty}^\infty e^{i\omega t}d\hat{r} .
\end{equation}
Each of the linear estimators in \eqref{yu}  can be represented by a linear filter which we symbolically write
\begin{subequations}\label{FuHy}
           \begin{align}
    y &= F(s)u + v,\\
    u & = H(s)y + r ,
 \end{align}
  \end{subequations}
 where $F(s)$ and $H(s)$ are proper rational transfer functions of dimensions $p\times q$ and $q\times p$, respectively. Hence we have the feedback configuration depicted in Fig.~\ref{figurelabel}.
 \begin{figure}
      \centering
      \includegraphics[scale=0.45]{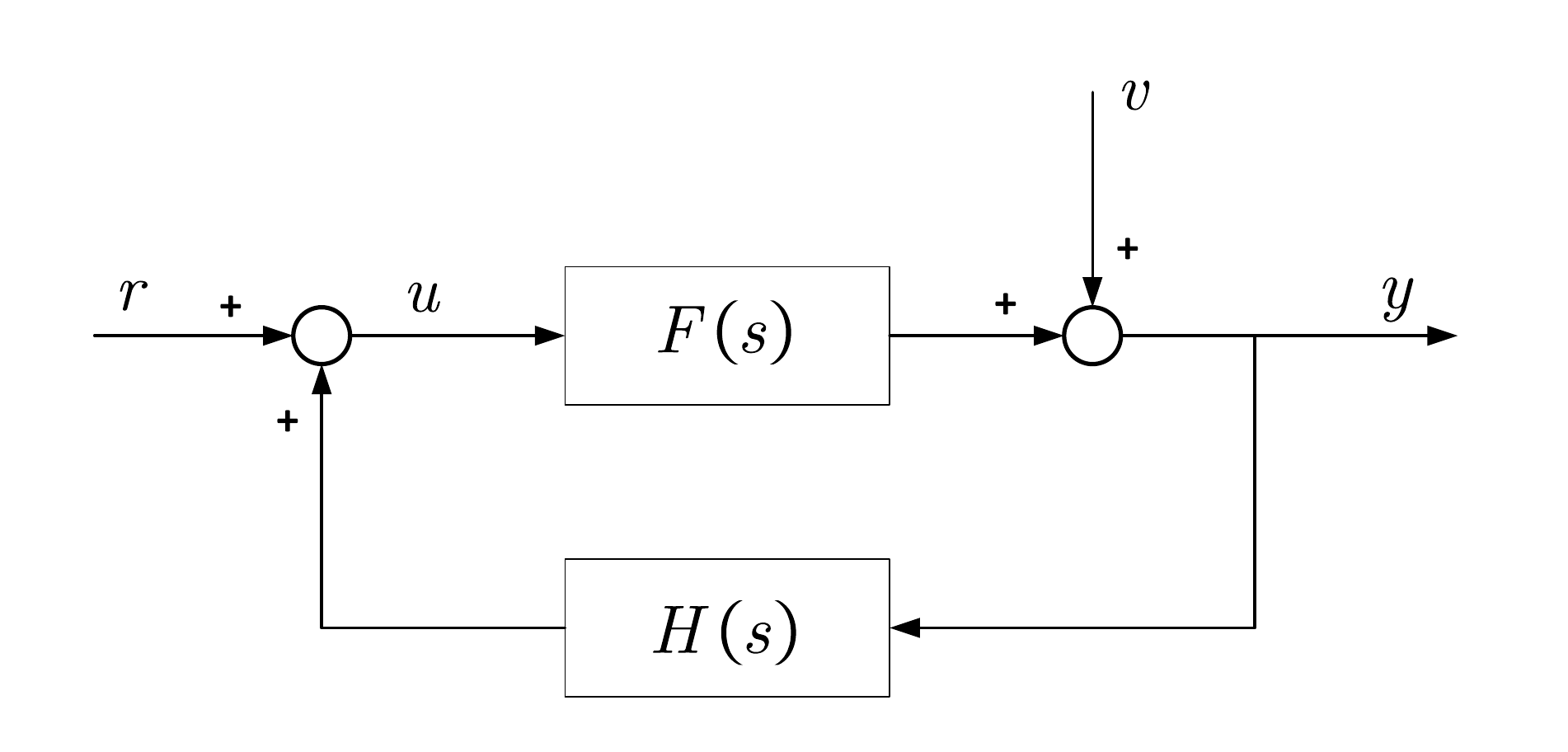}
      \caption{Block diagram illustrating  a feedback model}
      \label{figurelabel}
\end{figure}
Note that the transfer functions  $F(s)$ and $H(s)$ are in general not stable, but the overall feedback configuration needs to be internally stable \cite{DFT}, since that is needed for  all processes to be jointly stationary. The following results are essentially continuous-time versions of results in \cite{Caines}, where however a full-rank requirement to insure uniqueness was imposed.  Here we want to allow for rank-deficient spectral densities.

  \begin{theorem}\label{lem1}
        The transfer function matrix $T(s)$ from $\begin{bmatrix}v\\r\end{bmatrix}$ to $\begin{bmatrix}y\\u\end{bmatrix}$ in the feedback model \eqref{FuHy} is given by
\begin{subequations}\label{TPQ}
\begin{equation}
\label{T}
T(s)=\begin{bmatrix}P(s)&P(s)F(s)\\Q(s)H(s)&Q(s)\end{bmatrix},
\end{equation}
where
\begin{equation}
  \label{PQ}
\begin{split}
P(s)&=(I - F(s)H(s))^{-1},    \\
Q(s)&= (I - H(s)F(s))^{-1}.
\end{split}
\end{equation}
\end{subequations}
 Moreover,
 \begin{equation}
\label{PQFH}
P(s)F(s)=F(s)Q(s), \quad H(s)P(s)=Q(s)H(s),
\end{equation}
 and $T(s)$  is full rank and (strictly) stable.
 \end{theorem}

  \medskip

\begin{proof}
The  feedback system in Fig.~\ref{figurelabel} must be internally stable \cite[p. 37]{DFT} since the stationary processes $v$ and $r$ produce stationary processes $y$ and $u$. Hence $T(s)$ is (strictly) stable.
From \eqref{FuHy} we have
\begin{displaymath}
 \begin{bmatrix}y\\ u\end{bmatrix}=\begin{bmatrix}0 & F(s)\\H(s) & 0\end{bmatrix}\begin{bmatrix}y\\ u\end{bmatrix}+ \begin{bmatrix}v\\ r\end{bmatrix}
\end{displaymath}
         and therefore
\begin{displaymath}
N(s)\begin{bmatrix}y\\ u\end{bmatrix}=\begin{bmatrix}v\\ r\end{bmatrix} ,
\end{displaymath}
               where
 \begin{displaymath}
 N(s) :=\begin{bmatrix}I & -F(s)\\-H(s) & I\end{bmatrix}
\end{displaymath}
Now $(I - H(s)F(s))^{-1}$ is the transfer function from $r$ to $u$, and due to internal stability it must be stable. Consequently, $N(s)$ is invertible, and a straightforward calculation shows that $T(s)N(s)=I$ and hence
$T(s) = N(s)^{-1}$, as claimed. The relations \eqref{PQFH} are easy to verify.
 \end{proof}

 Although the errors $v$ and $r$ obtained by the procedure \eqref{yu} may be correlated, there exist feedback models where they are uncorrelated, and we shall assume that this is the case in the sequel. (We shall often suppress the argument $s$ whenever there is no risk of misunderstanding.) To see that this can be done, we provide the following construction. Define
 \begin{equation}
\label{ }
d\hat{w}=W(i\omega)^\dag \begin{bmatrix}d\hat{y}\\d\hat{u}\end{bmatrix},
\end{equation}
where $W^\dag =(W^*W)^{-1}W^*$ is the Moore-Penrose pseudo-inverse. Then, since $W^\dag W=I$, $d\hat{w}$ satisfies \eqref{dwhat}, so we have
\begin{equation}
\label{zetaspecrepr2}
\begin{bmatrix}d\hat{y}\\d\hat{u}\end{bmatrix}=W(i\omega)d\hat{w},
\end{equation}
which is the same as \eqref{zetaspecrepr}. Next, we define the error processes \eqref{vr} via
\begin{equation}
\label{w2vr}
\begin{bmatrix}d\hat{v}\\d\hat{r}\end{bmatrix}= \begin{bmatrix} G(i\omega)&0\\0&K(i\omega)\end{bmatrix}d\hat{w},
\end{equation}
where $G$ is $p\times m_1$ and $K$ is $q\times m_2$ with $m_1\leq p$, $m_2\leq q$ and $m_1+m_2=m$. Clearly the error processes $v$ and $r$ will then be uncorrelated.  Moreover, from Theorem~\ref{lem1}, we have
\begin{displaymath}
W=\begin{bmatrix}P&PF\\QH&Q\end{bmatrix}\begin{bmatrix} G&0\\0&K\end{bmatrix}=
\begin{bmatrix}PG&PFK\\QHG&QK\end{bmatrix}  ,
\end{displaymath}
which, by \eqref{PQFH}, can be written
\begin{displaymath}
W=\begin{bmatrix} W_{11}&W_{12}\\W_{21}&W_{22}\end{bmatrix}=\begin{bmatrix}PG&FQK\\HPG&QK\end{bmatrix},
\end{displaymath}
from which it follows that
\begin{equation}
\label{WFH}
W_{12}(s)=F(s)W_{22}(s), \quad W_{21}(s)=H(s)W_{11}(s)  .
\end{equation}
If $\Phi_\zeta$ is full rank, i.e., $m=p+q$,  $m_1=p$ and $m_2=q$, then
\begin{equation}
\label{FHfullrank}
F=W_{12}W_{22}^{-1}, \qquad H=W_{21}W_{11}^{-1} .
\end{equation}
Likewise, if $m_2=q$, we can recover $F$ from \eqref{FHfullrank}, and, if $m_1=p$, we can recover $H$. In general, $F$ and $H$ can only be partially recovered from \eqref{WFH}. A parametrized family of solutions $F$ may be obtained by adding columns to  $W_{22}$ so that it becomes square and full rank and likewise adding parametrized columns to $W_{12}$, yielding a family of solutions via the first relation in \eqref{FHfullrank}. The same can be done for $H$. However, in the sequel, we shall proceed in a different way.

We have established that, to each stable, minimal spectral factor $W$ of $\Phi_\zeta$ there corresponds a feedback model \eqref{FuHy}, and
\begin{equation}
\label{Phivr}
\Phi_\zeta(s)=T(s)\begin{bmatrix}\Phi_v(s)&0\\0&\Phi_r(s)\end{bmatrix}T(s)^*,
\end{equation}
where $\Phi_v(s)$ and $\Phi_r(s)$ are the spectral densities of $v$ and $r$, respectively, and $^*$ denotes transpose conjugate. Since $T(i\omega)$ has full rank a.e., $\Phi_\zeta$ is rank deficient if and only if at least one of $\Phi_v$ or $\Phi_r$ is.

 \medskip

 \begin{theorem}\label{lem2}
Suppose $(H\Phi_vH^*+\Phi_r)$ is positive definite a.e. on the imaginary axis. Then
\begin{equation}
\label{F}
F=\Phi_{yu}\Phi_u^{-1} - \Phi_vH^*(H\Phi_vH^*+\Phi_r)^{-1}(I-HF),
\end{equation}
that is
\begin{equation}
\label{Fred}
F=\Phi_{yu}\Phi_u^{-1}
\end{equation}
if and only if $\Phi_v H^*\equiv 0$.
\end{theorem}

\medskip

\begin{proof}
Given \eqref{Phidecomp}, \eqref{TPQ} and \eqref{Phivr}, we have
\begin{align*}
   \Phi_{yu}&= P(\Phi_vH^* +F\Phi_r)Q^*=P\Phi_vH^*Q^* + FQ\Phi_rQ^*  \\
   \Phi_u  &=  Q(H\Phi_vH^*+\Phi_r)Q^*=HP\Phi_vH^*Q^*+Q\Phi_rQ^* ,
\end{align*}
where we have used \eqref{PQFH}, i.e., $PF=FQ$ and $HP=QH$.
Hence
\begin{displaymath}
 \Phi_{yu}-F \Phi_u  = \Phi_vH^*Q^*,
\end{displaymath}
from which \eqref{F} follows.
\end{proof}

\section{Deterministic dynamical models}\label{sec:determ}

It was shown in \cite{GLsampling} that there are $p+q-m$ deterministic dynamical relations between the entries of $\zeta$ if
\[
\rank(\Phi_\zeta(i\omega))=m, \mbox{ a.e., for }\omega\in\mathbb R
\]
so that we have a map such as in Figure~\ref{system}. Taking $q=m$, this may correspond to the situation that $\Phi_v=0$, which, by \eqref{Phivr}, requires that $\rank(\Phi_r)=m$. Also, by Theorem~\ref{lem2}, \eqref{Fred} holds.

Consequently,  we have a feedback model
\begin{subequations}\label{FuHyv=0}
           \begin{align}
    y &= F(s)u\\
    u & = H(s)y + r ,
 \end{align}
  \end{subequations}
where a nontrivial $H(s)$ will permit $F(s)$ to be unstable, as the feedback will stabilize the feedback loop.

For a realization of $F(s)$ we refer to Theorem 1 in \cite{GLsampling}, where  the standard matrix notation
\begin{displaymath}
C(sI-A)^{-1}B+D=:\left[\begin{array}{c|c}A&B\\\hline C&D\end{array}\right].
\end{displaymath}\
for proper transfer functions is used.

\medskip

\begin{theorem}[Georgiou-Lindquist \cite{GLsampling}] \label{thm:GLsampling}
With $W,A$, $B$ and $C$ as above, re-order the rows of $C$ and partition
\[
C=\left(\begin{matrix}C_1\\C_0\end{matrix}\right)
\]
so that $C_0B$ is $m\times m$ and invertible. Reorder in the same way the entries of $\zeta$,  and as in \eqref{eq:z}, let $y$ represent the first $p$ entries and $u$ the remaining $q=m$. Then,
\begin{equation}
W(s)= \left(\begin{matrix} F(s)\\I\end{matrix}\right)M(s)s^{-1},
\end{equation}
where
\begin{subequations}
\begin{align}\nonumber\\[-.05in]
F(s)&=\left[\begin{array}{c|c}\Gamma &B (C_0B)^{-1}\\\hline C_1\Gamma &C_1B(C_0B)^{-1}\end{array}\right],\label{Ffinal}\\
M(s)&=\left[\begin{array}{c|c}A&B\\\hline C_0A&C_0B\end{array}\right], \label{M} \mbox{ and}\\
\Gamma&=A-B(C_0B)^{-1}C_0A. \label{Gamma}
\end{align}
\end{subequations}
\end{theorem}

\medskip

For reasons of lack of  space,  only a short outline of a proof of Theorem~\ref{thm:GLsampling} was given in \cite{GLsampling}. For the benefit of the reader we shall therefore give a more complete proof, which will also allow us to set up notation.

\medskip

\begin{proof}
First note that
\begin{align*}
(sI-A)^{-1}= & s^{-1}(sI-A+A)(sI-A)^{-1} \\
 = & s^{-1}I +s^{-1}A(sI-A)^{-1}.
\end{align*}
Therefore, partitioning $W$ conformably with $C$, we have
\begin{equation}
\label{ }
W(s)=\left(\begin{matrix}N(s)\\M(s)\end{matrix}\right)s^{-1},
\end{equation}
where
\begin{subequations}\label{MN}
\begin{eqnarray}
M(s) & = & C_0B+C_0A(sI-A)^{-1}B \\
N(s) & = & C_1B+C_1A(sI-A)^{-1}B .
\end{eqnarray}
\end{subequations}
Now, in view of \eqref{eq:factorization},
\begin{equation}
\label{MN2Phi}
\begin{bmatrix} \Phi_y & \Phi_{yu}\\\Phi_{uy} &\Phi_u\end{bmatrix}=
\begin{bmatrix} NN^*& NM^*\\MN^* &MM^*\end{bmatrix} \frac{1}{|s|^2},
\end{equation}
and therefore, it follows from \eqref{Fred} that
\begin{equation}
\label{MN2F}
F(s)= N(s)M(s)^{-1},
\end{equation}
which yields
\begin{equation}
\label{Fprel}
F(s)=C_1\left[I+A(sI-A)^{-1}\right]BM(s)^{-1} .
\end{equation}
From  \eqref{M} and Lemma~\ref{inverse} in the Appendix we have
\begin{displaymath}
M(s)^{-1}=\left[\begin{array}{c|c}\Gamma&B(C_0B)^{-1}\\\hline -(C_0B)^{-1}C_0A&(C_0B)^{-1}\end{array}\right],
\end{displaymath}
and therefore, since $B(C_0B)^{-1}C_0A=A-\Gamma$ by \eqref{Gamma},
\begin{align*}
  BM(s)^{-1} &=\left[I-B(C_0B)^{-1}C_0A(sI-\Gamma)^{-1}\right]B(C_0B)^{-1}   \\
                   &  =(sI-A)(sI-\Gamma)^{-1}B(C_0B)^{-1},
\end{align*}
which together with \eqref{Fprel} yields
\begin{align*}
  F(s)  &=C_1s(sI-\Gamma)^{-1}B(C_0B)^{-1} \\
    &=C_1(sI-\Gamma +\Gamma) (sI-\Gamma)^{-1}B(C_0B)^{-1} \\
    & =C_1B (C_0B)^{-1} + C_1\Gamma(sI-\Gamma)^{-1}B(C_0B)^{-1},
\end{align*}
which is precisely \eqref{Ffinal}, as required.
\end{proof}

\medskip

\begin{remark}
Although $M$ and $N$ in \eqref{MN} clearly depend on the particular choice of stable spectral factor $W$, the representation \eqref{Ffinal} does not. Indeed, \eqref{Ffinal} is constructed from \eqref{MN2F}, and, in view of \eqref{MN2Phi},
\begin{displaymath}
NM^{-1}=NM^*(MM^*)^{-1}=\Phi_{yu}\Phi_u^{-1},
\end{displaymath}
which depends only on the spectral density $\Phi_\zeta$.
\end{remark}

\medskip

\begin{remark}\label{mnrem}
Since $C_0\in\mathbb{R}^{m\times n}$ and $B\in\mathbb{R}^{n\times m}$, it is necessary that $m\leq n$ for $C_0B$ to be nonsingular.
\end{remark}

\medskip

We also refer the reader to a recent paper \cite{Deistler19}, which deals with a theme akin to that in Theorem~\ref{thm:GLsampling}.

Next we would like to investigate under what conditions $F(s)$ is strictly stable, but let us first consider whether this will be prevented by the following result.

\medskip

\begin{lemma}\label{Gammarank}
Suppose $B(C_0B)^{-1}C_0$ has $n$ linearly independent eigenvectors. Then
\begin{displaymath}
\rank(\Gamma) = n-m.
\end{displaymath}
\end{lemma}

\medskip

\begin{proof}
By Lemma~\ref{similaritylem} in the Appendix, the nonzero eigenvalues of $B(C_0B)^{-1}C_0$ are the same as those of $C_0B(C_0B)^{-1} = I_{m}$, so $B(C_0B)^{-1}C_0$ has $m$ nonzero eigenvalues all equal to 1. Then there is an $n\times n$ matrix $P$ such that
\begin{displaymath}
P^{-1}B(C_0B)^{-1}C_0P= \begin{bmatrix}I_m&  \\&0_{n-m}\end{bmatrix},
\end{displaymath}
and consequently
\begin{displaymath}
P^{-1}(I-B(C_0B)^{-1}C_0)P=\begin{bmatrix}0_m&  \\&I_{n-m}\end{bmatrix}.
\end{displaymath}
Then, since $\Gamma=(I-B(C_0B)^{-1}C_0)A$ and $A$ has full rank for being a stability matrix, the lemma follows.
\end{proof}

\medskip

Consequently, since $m>0$, $\Gamma$ will be singular under the conditions of Lemma~\ref{Gammarank}. Reformulating  \eqref{Ffinal} we see that
\begin{equation}
\label{Falt}
F(s) = sC_1(sI - \Gamma)^{-1}B(C_0B)^{-1},
\end{equation}
and consequently at least one zero eigenvalue of $\Gamma$ will cancel, reducing the McMillan degree of $F(s)$. However, by the next theorem, in general  all zero eigenvalues will be cancelled, so the degree of $F(s)$ is reduced to the number of nonzero eigenvalues of $\Gamma$.

\medskip

\begin{theorem}\label{McMillanthm}
Suppose $B(C_0B)^{-1}C_0$ has $n$ linearly independent eigenvectors. Then $F(s)$, given by \eqref{Ffinal}, has McMillan degree at most $n-m$.
\end{theorem}

\medskip

\begin{proof}
The observability matrix of \eqref{Ffinal} is
\begin{displaymath}
\begin{bmatrix}
C_1\Gamma\\
C_1\Gamma^2\\
\vdots\\
C_1\Gamma^{p-m}
\end{bmatrix}
=
\begin{bmatrix}
C_1\\
C_1\Gamma\\
\vdots\\
C_1\Gamma^{p-m-1}
\end{bmatrix}
\Gamma ,
\end{displaymath}
which has at most the same rank as $\Gamma$. Hence, by Lemma~\ref{Gammarank}, the realization \eqref{Ffinal} is not observable, so it is not minimal. In fact, the dimension of the unobservable subspace is at least $m$, so there is a  realization of $F(s)$ with a dimension that has been reduced with $m$, i.e., from $n$ to $n-m$.
\end{proof}

\medskip

\begin{remark}
If $m = n > 0$, $B$ and $C_0$ are both $m\times m$ matrices, which, in view of the condition rank$(C_0B)=m$,  must be invertible. Then $ \Gamma = (I - B(C_0B)^{-1}C_0)A = 0$, and consequently
    $$F(s) = C_1C_0^{-1}$$
is constant. In this case, all eigenvalues of $\Gamma$ are zero, but all are canceled in forming $F(s)$, resulting in a strictly stable $F(s)$.
\end{remark}

\section{Granger causality and stability}\label{sec:Granger}

Suppose that we want to predict the future of $y$ given the past of $y$, would we get a better estimate if we also would know the past of $u$? If the answer is affirmative, we have {\em Granger causality from $u$ to $y$} \cite{Granger63,Granger69,BS17,PD12,ADD19}. In mathematical terms this can be written \cite[Definition 1]{Granger69}
\begin{equation}
\label{ }
\mathbb{E}^{\Hb_t^-(y)\vee\Hb_t^-(u)}\lambda\ne\mathbb{E}^{\Hb_t^-(y)}\lambda \quad \text{for $\lambda\in\Hb_t^+(y)$},
\end{equation}
where $\mathbb{E}^{\mathbf{A}}\lambda$ denotes the orthogonal projection of $\lambda$ onto the subspace $\mathbf{A}$ and $\vee$ is vector sum, i.e., $\mathbf{A}\vee\mathbf{B}$ is the closure in the Hilbert space of stochastic variables of the sum of the subspaces $\mathbf{A}$ and $\mathbf{B}$; see, e.g., \cite{LPbook}.

It is easier to analyze the negative statement. We have non-causality from $u$ to $y$ in the sense of Granger if
\begin{equation}
\label{noncausality}
\mathbb{E}^{\Hb_t^-(y)\vee\Hb_t^-(u)}\lambda =\mathbb{E}^{\Hb_t^-(y)}\lambda \quad \text{for $\lambda\in\Hb_t^+(y)$},
\end{equation}
which we may also write, with $A\ominus B$ the orthogonal complement of $B\subset A$ in $A$,
\begin{displaymath}
\mathbb{E}^{\Hb_t^-(y)}\lambda +\mathbb{E}^{[\Hb_t^-(y)\vee\Hb_t^-(u)]\ominus\Hb_t^-(y)}\lambda =\mathbb{E}^{\Hb_t^-(y)}\lambda
\end{displaymath}
for all $\lambda\in\Hb_t^+(y)$, which is equivalent to
\begin{displaymath}
\left[\Hb_t^-(y)\vee\Hb_t^-(u)\right]\ominus\Hb_t^-(y)\perp \Hb_t^+(y),
\end{displaymath}
where $\mathbf{A}\perp\mathbf{B}$ means that the subspaces $\mathbf{A}$ and $\mathbf{B}$ are orthogonal. Then, using the equivalence between properties (i) and (v) in \cite[Proposition 2.4.2]{LPbook}, we see that this in turn is equivalent to
\begin{equation}
\label{condorth}
\Hb_t^-(u)\perp \Hb_t^+(y)\mid \Hb_t^-(y),
\end{equation}
i.e., $\Hb_t^-(u)$ and $\Hb_t^+(y)$ are conditionally orthogonal given $\Hb_t^-(y)$.

Hence \eqref{condorth} is geometric condition for lack of Granger causality. Now returning to the feedback model, and in particular to \eqref{u2y}, we see from \eqref{condorth} that once the past of $y$ is known, the future of $y$ is uncorrelated to the past of $u$, and therefore $\mathbb{E}\{y(t)\mid \Hb_t^-(u)\}=0$, so lack of Granger causality is equivalent to $F(s)\equiv 0$. Conversely, we have Granger causality from $u$ to $y$ if and only if $F(s)$ is nonzero.

Applying an analogous argument to \eqref{y2u}, we see that the geometric conditon
\begin{equation}
\label{feedbackfree}
\Hb^-(y)\perp \Hb^+(u)\mid \Hb^-(u) ;
\end{equation}
is equivalent to $H(s)\equiv 0$. Then there is no feedback from $y$ to $u$  \cite[p. 677]{LPbook}. Consequently, as stressed in \cite{Caines}, Granger causality and feedback are dual concepts.

In the setting of Section III we must have $H(s)$ nonzero if $F(s)$ is not strictly stable, because it is needed for stabilization of the feedback loop. Conversely, if $H(s)$ is zero, $F(s)$ must be strictly stable.

\medskip

\begin{theorem}
Consider the feedback model \eqref{FuHy}, and in particular, \eqref{FuHyv=0}. Then there is causality from $u$ to $y$ in  the sense of Granger if and only if $F(s)$ is nonzero, and there is no feedback from $y$ to $u$ if and only if $H(s)$ is identically zero. In this case $F(s)$ is (strictly) stable.
\end{theorem}

\medskip

It could be argued that a better (and stronger) definition of causality of $F(s)$ in the present setting is  \eqref{feedbackfree}, namely that is there is no feedback from $y$ to $u$ \cite[Section~2.6.5]{LPbook}.

\section{Examples}\label{sec:examples}

\subsection{Example 1}

Let $\Phi_\zeta(s)$ be a spectral density  \eqref{Phizeta} with
$$
    A = \begin{bmatrix}
        -9 & -4 & -6\\
        6 & 1& 6\\
        4 & 2 &2
        \end{bmatrix},\;
    B = \begin{bmatrix}
            0 \\ 4 \\ -4
        \end{bmatrix},\;
    C = \begin{bmatrix}
            1 & 1 & 0\\
            2& 2& 1\\
            0& 0& 1\\
            3 & 1 & 2
        \end{bmatrix},
    $$
which has rank $m=1$. First take $C_0$ to be the first row of $C$, i.e., $u=\zeta_1$,  $y=(\zeta_2,\zeta_3,\zeta_4)'$, and
$$
        C_0 = \begin{bmatrix}
            1 & 1 & 0
        \end{bmatrix},\quad
        C_1 = \begin{bmatrix}
            2& 2& 1\\
            0& 0& 1\\
            3 & 1 & 2
        \end{bmatrix}.
    $$
Then, $C_0B=4$, $B(C_0B)^{-1}C_0$ has $n=3$ independent eigenvectors, and
    $$
        \Gamma = (I - B(C_0B)^{-1}C_0)A
            = \begin{bmatrix}
            -9 & -4 & -6\\
            9 & 4 & 6\\
            1 & -1 & 2 \end{bmatrix},
    $$
which has rank two with eigenvalue $0$, $-1$ and $-2$.  However, by Theorem~\ref{McMillanthm}, the pole at zero will cancel, and we obtain
 \begin{displaymath}
  F(s) = \frac{1}{(s+2)(s+1)}\begin{bmatrix}
                    s^2-9 \\ -s^2-6s-13 \\ -s^2-5s+4
                    \end{bmatrix},
\end{displaymath}
which is strictly stable of degree two rather than three.

Since the McMillan degree of $F(s)$ is two, it has a minimal realization of dimension two. One such realization is given by
  \begin{displaymath}
        F(s)=\left[\begin{array}{c|c} \Gamma_F &B_F\\\hline C_F& D_F\end{array}\right],
    \end{displaymath}
   where
    $$\Gamma_F = \begin{bmatrix} -2 & 0\\ 0 & -1\end{bmatrix},
        B_F = \begin{bmatrix} 1\\1 \end{bmatrix},
        C_F = \begin{bmatrix} 5 & -8\\ 5 & -8\\ -10 & 8 \end{bmatrix},$$
        $$D_F = \begin{bmatrix} 1 & -1 & -1\end{bmatrix}'.$$

Next, take $C_0$ to be the second row of $C$, i.e., $u=\zeta_2$,  $y=(\zeta_1,\zeta_3,\zeta_4)'$, and
$$
        C_0 = \begin{bmatrix}
            2 & 2 & 1
        \end{bmatrix},
        C_1 = \begin{bmatrix}
            1& 1& 0\\
            0& 0& 1\\
            3 & 1 & 2
        \end{bmatrix}.
    $$
Then, $C_0B=4$ and
    $$
        \Gamma = (I - B(C_0B)^{-1}C_0)A
            = \begin{bmatrix}
            -9 & -4 & -6\\
            8 & 5 & 4\\
            2 & -2 & 4 \end{bmatrix},
    $$
which has rank two with eigenvalue $0$, $-3$ and $3$. In this case, $B(C_0B)^{-1}C_0$ has only two independent eigenvalues, so we cannot apply Theorem~\ref{McMillanthm}. However, due to \eqref{Falt}, the zero pole will nevertheless cancel, and we obtain
 \begin{displaymath}
    F(s) = \frac{1}{(s+3)(s-3)}\begin{bmatrix}
                    s^2+3s+2 \\ -s^2-6s-13 \\ -s^2-5s+4
                    \end{bmatrix},
  \end{displaymath}
  which is unstable. Then to stabilize the closed-loop system a nonzero $H(s)$ is required. As such an $H(s)$ will not be unique, we are only interested in its existence. The system
  $$\Gamma_F = \begin{bmatrix} -3 & 0\\ 0 & 3\end{bmatrix},
        B_F = \begin{bmatrix} \frac{1}{3}\\ \frac{1}{3} \end{bmatrix} = \frac{1}{3}\begin{bmatrix} 1\\ 1 \end{bmatrix},
        C_F = \begin{bmatrix} -1 & 10\\ 2 & -20\\ -5 & -10 \end{bmatrix},$$
        $$D_F = \begin{bmatrix} 1 & -1 & -1\end{bmatrix}'$$
is a minimal realization of $F(s)$.

\subsection{Example 2: A counterexample}

In Example 1 we had a dynamical system with stable transfer function $F(s)$ and another with an unstable one.
Manfred Deistler \cite{Deistler} has recently posed the question whether there is always a selection of inputs, such that corresponding transfer function is stable and causal. The following simple counterexample answers this question in  the negative.

Let $\Phi_\zeta$ be a spectral density  with
$$
  A = \begin{bmatrix}
        -3 & -\frac{4}{3}\\
        \frac{3}{2} & 0
        \end{bmatrix},
    B = \begin{bmatrix}
            -3 \\ -2
        \end{bmatrix},
    C = \begin{bmatrix}
            1 & 0\\
            1& -1
        \end{bmatrix},
    $$
which has rank $m=1$. First take $C_0$ to be the first row of $C$, i.e., $u=\zeta_1$  $y=\zeta_2$, and
$$
        C_0 = \begin{bmatrix}
            1 &  0
        \end{bmatrix},
        C_1 = \begin{bmatrix}
            1& -1
        \end{bmatrix}.
    $$
 Then,
    $$
   \Gamma = (I - B(C_0B)^{-1}C_0)A
            = \begin{bmatrix}
            0 & 0 \\
            \frac{7}{2} & \frac{8}{9} \end{bmatrix},
    $$
which has rank 1 with eigenvalue $0$ and $\frac{8}{9}$; and
\begin{displaymath}
  F(s) = \frac{6s-79}{2(9s-8)},
\end{displaymath}
which is unstable.

Next take $C_0$ to be the second row of $C$, i.e., $u=\zeta_2$  $y=\zeta_1$, and
$$
        C_0 = \begin{bmatrix}
            1 & -1
        \end{bmatrix},
        C_1 = \begin{bmatrix}
            1& 0
        \end{bmatrix}.
    $$
Then,
    $$
   \Gamma = (I - B(C_0B)^{-1}C_0)A
            = \begin{bmatrix}
            \frac{21}{2} & \frac{8}{3} \\
            \frac{21}{2} & \frac{8}{3} \end{bmatrix},
    $$
which has rank 1 with eigenvalue $0$, $\frac{79}{6}$ and
 \begin{displaymath}
    F(s) = \frac{2(9s-8)}{(6s-79)}.
  \end{displaymath}

Consequently, no matter how we choose $C_0$, there is no selection admitting a stable $F(s)$.

\section{A remark on the indirect method of CARMA identification}\label{sec:CARMA}

Continuous-time systems \eqref{eq:continuous} are generally identified from sampled data. Instead of measuring the continuous signal $\zeta$ itself, in practice we observe the sampled process $\zeta_k=\zeta(kh)$ for some sampling period $h$. The corresponding state process $x$ is then sampled as $x_k=x(kh)$ and satisfies the recursion
\[
x_{k+1}=e^{Ah}x_k+\int_{kh}^{(k+1)h}e^{A((k+1)h-\tau)}Bdw(\tau).
\]
which leads to the the discrete-time stochastic system
\begin{subequations}\label{eq:xk}
\begin{align}
x_{k+1}&= A_d x_k +B_d v_k\\
\zeta_k&=C_d x_k
\end{align}
\end{subequations}
where $A_d=e^{Ah}$, $C_d=C$ and the term
\begin{displaymath}
B_d v_k=\int_{kh}^{(k+1)h}e^{A((k+1)h-\tau)}Bdw(\tau)
\end{displaymath}
defines a sequence of independent random vectors having zero mean, which is normalized by the matrix $B_d $ so that $v_k$ becomes normalized white noise. It is then straightforward to see that the covariance matrix $P:=\mathbb{E}\{x(t)x(t)'\}$ satisfies the two Lyapunov equations
\begin{subequations}
\begin{align}
\label{}
    &AP+PA'+BB'=0   \\
    & P=A_dPA_d +B_dB_d
\end{align}
\end{subequations}
simultaneously; see e.g., \cite{GLsampling}.

The main point in \cite{GLsampling} was to demonstrate that a rank deficiency in the continuous-time spectral density $\Phi_\zeta$ is lost in the discrete time observation but can nevertheless be recovered by inverting the sampling process.  Indeed, due to reachability,
\begin{displaymath}
B_dB_d' = \int _0^h e^{A(h-\sigma)}BB'e^{A'(h-\sigma)}d\sigma
\end{displaymath}
is always full rank even if $BB'$ is not. This hides any rank deficiency in the continuous-time spectral density.  We quote the following theorem from  \cite{GLsampling}.

\medskip

\begin{theorem}[Georgiou-Lindquist \cite{GLsampling}]\label{prop:discretemodel}
Consider the continuous-time stochastic model \eqref{eq:continuous} having parameters $(A,B,C)$. Sampling with period $h$ gives rise to the discrete-time stochastic model \eqref{eq:xk} with parameters $(A_d,B_d,C_d)$. Conversely, if the parameters $(A_d,B_d,C_d)$ of the
stochastic model \eqref{eq:xk} (with $A_d$ a stability matrix) are such that
\begin{subequations}\label{conditions}
\begin{align}
\mbox{i) }&A_d \mbox{ admits a (principal) matrix logarithm \cite{culver1966existence}}
\label{eq:logm}\\
& \mbox{that we denote }\log(A),\nonumber\\
\mbox{ii) }&\det (B_dB_d')\neq 0, \label{Bsqinv}\\
\label{eq:key}
\mbox{iii) }&\log(A)P+P\log(A)'\leq 0,\mbox{ for }P\\
&\mbox{ the solution of } P=A_dPA_d'+B_dB_d',\nonumber
\end{align}
\end{subequations}
then \eqref{eq:xk} arises by sampling \eqref{eq:continuous} with $A=\frac{1}{h}\log(A_d)$, $C=C_d$,
and $B$  a left factor of $-(AP+PA')$ of full column rank.
\end{theorem}

\medskip

Except for condition \eqref{Bsqinv}, the inverse procedure of Theorem~\ref{prop:discretemodel} has been standard in various forms in classical literature on CARMA identification \cite{Soderstrom91,MahataFu07,BPS13}. This was overlooked in \cite{GLsampling}. In fact, both S\"oderstr\"om \cite{Soderstrom91} and Mahata and Fu \cite{MahataFu07} present a number of different algorithms for identifying a continous-time models \eqref{eq:continuous} from sampled data via the discrete-time model \eqref{eq:xk}.

However, the papers \cite{Soderstrom91,MahataFu07} appear to have missed the point \eqref{Bsqinv}, which turns out to be crucial, and this is our main point here. In fact, the discrete-time sampled model does not inherit the rank deficiency discussed in \cite{GLsampling} and in our present paper. For example, in algorithm E2 in  \cite{Soderstrom91} the matrix $R_d=B_dB_d'$ is assumed to be of rank 1, whereas it must be of full rank if \eqref{eq:xk} is truly a sampled model.

\section{Conclusions}\label{sec:conclusions}

To explain the stability and causality properties of hidden dynamical relations in stationary stochastic networks, we have introduced feedback models and related this to Granger causality. Deterministic dynamical relations occur because of rank deficiencies in continuous-time spectral densities. As established in \cite{GLsampling} these rank deficiencies disappear after sampling. This naturally connects to the indirect method of CARMA identification, where the continuous-time model is identified based on sampled data.

\appendix

\begin{lemma}\label{inverse}
Let
\[
P(s):=\left[\begin{array}{c|c}A&B\\\hline C&D\end{array}\right],
\]
where $D$ is square and nonsingular. Then
\[
P(s)^{-1}:=\left[\begin{array}{c|c}A-BD^{-1}C&BD^{-1}\\\hline -D^{-1}C&D^{-1}\end{array}\right].
\]
\end{lemma}

\begin{proof}
The rational matrix function $P$ is the transfer function of the control system
\begin{displaymath}
\begin{cases}
\dot{x}=Ax+Bu\\
y=Cx+Du
\end{cases}.
\end{displaymath}
Solving for $u$, we have
\begin{displaymath}
u=D^{-1}(y-Cx),
\end{displaymath}
which inserted in the first equation yields
\begin{displaymath}
\dot{x}=(A-BD^{-1}C)x+BD^{-1}y,
\end{displaymath}
and hence we have the inverse system
\begin{displaymath}
\begin{cases}
\dot{x}=(A-BD^{-1}C)x+BD^{-1}y\\
u=-D^{-1}Cx + D^{-1}y
\end{cases}
\end{displaymath}
with transfer function $P(s)^{-1}$.
\end{proof}

\medskip

\begin{lemma}\label{similaritylem}
    If $A\in \mathbb{R}^{m\times n}$, $B\in \mathbb{R}^{n\times m}$, then the nonzero eigenvalues of $AB$ and $BA$ are the same.
\end{lemma}

     \medskip
\begin{proof}
Setting
\begin{displaymath}
T_1:=\begin{bmatrix}AB & 0\\ B & 0\end{bmatrix}\quad\text{and}\quad
T_2:=\begin{bmatrix}0 & 0\\ B & BA\end{bmatrix},
\end{displaymath}
 we observe that
 \begin{displaymath}
S^{-1}T_1S =T_2,
\end{displaymath}
where  $S$ is the $(m+n)\times (m+n)$ matrix
\begin{displaymath}
S:= \begin{bmatrix} I_m & A\\ 0 & I_n \end{bmatrix},
\end{displaymath}
i.e., $T_1$ and $T_2$ are similar. Consequently they have the same characteristic polynomial. That is to say,
\begin{displaymath}
\det(\lambda I_m - AB)\lambda^n=\lambda^m\det(\lambda I_n - BA)
\end{displaymath}
Then, for each nonzero eigenvalue $\tilde\lambda$ of $AB$, we have $\det(\tilde\lambda I_m - AB)=0$, so $\tilde\lambda$ is also an eigenvalue of $BA$. For the same reason each nonzero eigenvalue of $BA$ is an eigenvalue of $AB$.
\end{proof}

\addtolength{\textheight}{-3cm}   

\end{document}